\documentclass[a4paper,openbib,12pt]{article}
\usepackage{geometry}
\usepackage{amsmath}
\usepackage{amssymb}
\usepackage{amsfonts}

\newtheorem{theorem}{Theorem}[section]

\newtheorem{corollary}[theorem]{Corollary}

\newtheorem{definition}[theorem]{Definition}

\newtheorem{lemma}[theorem]{Lemma}

\newtheorem{question}[theorem]{Question}

\newtheorem{remark}[theorem]{Remark}

\newenvironment{proof}[1][Proof]{\noindent\textbf{#1.} }{\ \rule{0.5em}{0.5em}}
\begin{document}

\title{Solvable subgroup theorems, length functions and topological entropy}
\author{Shengkui Ye}
\maketitle

\begin{abstract}
We prove a general solvable subgroup theorem in terms of length functions.
As applications, we obtain a solvable subgroup theorem in dynamical systems:
any solvable group of finite Hirsch length acting on a smooth manifold with
uniformly positive topological entropies must be virtually $\mathbb{Z}^{n}.$
Some applications to group actions on cubical complexes and injective
spaces, interval exchange transformations, and RFRS groups are also
discussed.

\end{abstract}

\section{Introduction}

The solvable subgroup theorems were established in many branches of
mathematics by many people, eg. Gromoll-Wolf \cite{gw} and Yau \cite{yau}
for fundamental groups of non-positively curved manifolds, Bridson-Haefliger 
\cite{bh} for groups acting semi-simply on CAT(0) spaces,
Farb-Lubotzky-Minsky \cite{flm} for mapping class groups,
Bestvina-Feighn-Handel \cite{bfh} for outer automorphism groups of free
groups, Gersten-Short \cite{gs} for biautomatic groups, Conner \cite{conner}
for groups with positive stable norms coming from left-invariant metrics,
Prytu\l a \cite{p} for groups acting geometrically on systolic complexes,
and so on. Basically, the theorems say that solvable groups acting nicely on
nice spaces are special (eg. virtually abelian or virtually
abelian-by-abelian). In this note, we prove a general Solvable Subgroup
Theorem in terms of length functions defined by the author in \cite{ye}.

Let $G$ be a group. A real-valued function $l:G\rightarrow \lbrack 0,\infty
) $ is called a length function if

(1) $l(g^{n})=|n|l(g)$ for any $g\in G$ and $n\in \mathbb{Z};$

(2) $l(hgh^{-1})=l(g)$ for any $h,g\in G;$ and

(3) $l(ab)\leq l(a)+l(b)$ for \emph{commuting} elements $a,b.$

Length functions exist in many branches of mathematics, eg. stable word
lengths, stable norms, smooth measure-theoretic entropy, translation lengths
on $\mathrm{CAT}(0)$ spaces and Gromov $\delta $-hyperbolic spaces, stable
norms of quasi-cocycles, rotation numbers of circle homeomorphisms,
dynamical degrees of birational maps, absolute values of Margulis invariants
(cf. \cite{ye}, Section 2) and filling volumes (cf. \cite{bf}). A length
function $l$ is called purely positive if $l(g)>0$ for each infinite-order
element $g.$ If there is a constant $c>0$ such that $l(g)>c$ for any nonzero 
$l(g),$ the length function $l$ is called discrete. A group $G$ is called
purely positive or discrete if there exists a length function $l$ on $G$
with the corresponding properties. A group $G$ is $\mathbb{C}$-linear if $G$
can be embedded into the general linear group $\mathrm{GL}_{n}(\mathbb{C})$
for some integer $n.$ Recall that a group virtually has a property $P$ if
some finite-index subgroup has the property $P.$ The following are general
solvable subgroup theorems in terms of length functions.

\begin{theorem}
\label{th1}Let $G$ be a finitely generated solvable group with a purely
positive length function. Suppose that 

\begin{enumerate}
\item[i)] $G$ is virtually torsion-free, and

\item[ii)] any abelian subgroup $A$ has finite rational rank, that is $\dim _{\mathbb{Q}%
}(A\bigotimes_{\mathbb{Z}}\mathbb{Q})<+\infty $.

Then $G$ has a finite-index subgroup $K$ that is either abelian or isomorphic to a non-nilpotent $\mathbb{C}$-linear
subgroup of $H\rtimes \mathbb{Z}^{n}$, for some integer $n$ and a finitely
generated $\mathbb{Z}[\mathbb{Z}^{n}]$-module $H.$
\end{enumerate}

\end{theorem}

Neither the virtual torsion-freeness nor the finiteness of the dimension $%
\dim _{\mathbb{Q}}(A\bigotimes_{\mathbb{Z}}\mathbb{Q})$ in Theorem \ref{th2}
can be dropped. For each integer $n\geq 2,$ let $H(\mathbb{Z}/n)$ the group
consisting of all $3\times 3$ strictly upper triangular matrices with
entries in $\mathbb{Z}/n.$ The wreath product $G_{1}=H(\mathbb{Z}/n))\wr 
\mathbb{Z=}(\bigoplus_{i\in \mathbb{Z}}H(\mathbb{Z}/n))\rtimes \mathbb{Z}$
is finitely generated solvable with finite $\dim _{\mathbb{Q}}(A\bigotimes_{%
\mathbb{Z}}\mathbb{Q})$ for any abelian subgroup $A,$ but not virtually
torsion-free. Let $A\in \mathrm{GL}_{n}(\mathbb{Z})$ be an irreducible $%
\mathbb{C}$-diagonalizable matrix with a norm-one eigenvalue. Then the
stable word length on the semidirect product $\mathbb{Z}^{n}\rtimes _{A}%
\mathbb{Z}$ is purely positive (see Lemma \ref{stableword}). The wreath
product $G_{2}=(\mathbb{Z}^{n}\rtimes _{A}\mathbb{Z)}\wr \mathbb{Z=}%
(\bigoplus_{i\in \mathbb{Z}}\mathbb{Z}^{n}\rtimes _{A}\mathbb{Z})\rtimes 
\mathbb{Z}$ is torsion-free finitely generated solvable. Note that both $%
G_{1},G_{2}$ are purely positive (cf. \cite{conner}, Example 7.2), but
neither virtually abelian nor virtually abelian-by-abelian.

\begin{theorem}
\label{th3}Every solvable subgroup of finite virtual cohomological dimension
in a discretely purely positive group is a finite extension of $\mathbb{Z}%
^{n}.$
\end{theorem}

Note that the group $\mathbb{Q}$ of rational numbers is solvable, but not a
finite extension of $\mathbb{Z}^{n}.$ Moreover, the wreath product $\mathbb{Z%
}\wr \mathbb{Z=}(\bigoplus \mathbb{Z})\rtimes \mathbb{Z}$ is finitely
generated solvable having a discrete stable word length (cf. \cite{conner},
Example 7.2). This means that neither the discreteness nor finiteness of the
virtual cohomological dimension in Theorem \ref{th3} can be dropped.

In the theory of dynamical systems, Hu-Shi-Wang \cite{hsw} proves that a
Heisenberg group acting on smooth Riemannian manifolds is restrictive, in
the sense that the central elements must have vanishing Lyapunov exponents
and topological entropy. We obtain a general Solvable Subgroup Theorem in
dynamic systems theory.

\begin{theorem}
\label{th-1}
Let $G$ be a group consisting of $C^\infty$-diffeomorphisms of a closed Riemannian manifold $M$ such that each non-identity element has positive topological entropy. Then each solvable subgroup $H<G$ of finite virtual cohomological dimension is virtually abelian-by-abelian. Furthermore, if the topological entropies have a uniform positive lower bound, then $H$ is virtually a finitely generated abelian group $\mathbb Z^n$. The same conclusions hold under the following alternative hypothesis: for every non-identity element $f\in G$, there exist $x\in M$ and $0\neq u\in T_xM$ such that the Lyapunov exponent \[ \chi(x,u;f) = \lim_{n\to\infty} \frac{1}{n} \log \frac{\|D_xf^n(u)\|}{\|u\|} \] exists and is positive. Under this alternative hypothesis, the second conclusion holds if these positive Lyapunov exponents have a uniform positive lower bound.

\end{theorem}

For concrete group actions on manifolds$,$ Theorem \ref{th-1} can be
strengthened. For example, any solvable subgroup of the special linear group 
$\mathrm{SL}_{n}(\mathbb{Z})$ acting on the torus $T^{n}=\mathbb{R}^{n}/%
\mathbb{Z}^{n}$ consisting of elements with positive topological entropy is
virtually abelian. Any solvable subgroup with a finite virtual cohomological
dimension of the polynomial automorphism group $\mathrm{Aut}(\mathbb{C}^{2})$
acting on the sphere $S^{4}=\mathbb{C}^{2}\cup \{\infty \}$ consisting of
elements with positive topological entropy is virtually abelian. See Section
2 for more details.

The following question was asked by Hu-Shi-Wang \cite{hsw} (Question 1.3).

\begin{question}
Let $H=\langle f,g,h\mid fh=hf,gh=hg,[f,g]=h\rangle $ be the Heisenberg
group, acting on a compact $C^{\infty }$-Riemannian manifold $M$ preserving
a Borel probability measure $\mu .$ Is the norm $\Vert Dh_{x}^{n}\Vert $
bounded by $e^{\sqrt{n}\varepsilon }$ for some $\varepsilon >0,$ or even by
a polynomial in $n$ for $\mu $-a.e. $x\in M$?
\end{question}

We give a positive answer to this question in the case of exponential
function.

\begin{theorem}
\label{th5}Let $H=\langle f,g,h\mid fh=hf,gh=hg,[f,g]=h\rangle $ be the
Heisenberg group acting on a compact $C^{\infty }$-Riemannian manifold $M.$
There exist constants $K,C>0$ such that 
\begin{equation*}
\log \Vert Dh_{x}^{n}\Vert \leq K\sqrt{n}+C
\end{equation*}%
for any $x\in M.$
\end{theorem}

Some versions of Theorem \ref{th1} and Theorem \ref{th3} are obtained by
Conner \cite{conner} for length functions coming from semi-norms on groups.
The strategy of our proofs is also similar to that of Conner \cite{conner}.
However, many length functions are not from semi-norms (see \cite{ye},
Section 2, or the final section of the current paper for example). A typical
length function not from an obvious semi-norm is the topological entropy (on
smooth manifolds). The proofs in \cite{conner} are essentially based on
semi-norms, which do not exist for general length functions assuming subadditivity only for commuting pairs. To deal with
this difficulty, we have to study the extension of length functions from $%
\mathbb{Z}^{n}$ to $\mathbb{R}^{n},$ that are compatible with the action of
an integer matrix (see Section \ref{sect}). The remaining part of the proofs
is essentially the same as that of Conner \cite{conner} \cite{con97}. For
completeness, we give detailed and self-contained proofs. The proof of
Theorem \ref{th-1} might follow \cite{conner} with some efforts on
modification of the definition of seminorms. The interest here is to bring
ideas from group theory to solve problems in dynamical system. It is
probably that Theorem \ref{th-1} can be extended to non-compact manifolds, but we
don't have a proof yet. In the final section, we give further applications
to group actions on cubical complexes and injective spaces, interval
exchange transformations, and RFRS groups.

\section{Semidirect product\label{sect}}

Let $A\in \mathrm{GL}_{n}(\mathbb{Z})$ be a matrix and $G=\mathbb{Z}%
^{n}\rtimes _{A}\mathbb{Z}$ the semidirect product.

\begin{lemma}
\label{key}For any length function $l^{\prime }:G\rightarrow \mathbb{R}%
_{\geq 0},$ there is a continuous length function $l:\mathbb{R}^{n}\mathbb{%
\rightarrow R}_{\geq 0}$ satisfying $l_{\mathbb{Z}^{n}}=l_{\mathbb{Z}%
^{n}}^{\prime }$ and the following properties:

1) (subadditive) $l(r_{1}+r_{2})\leq l(r_{1})+l(r_{2})$ for any $%
r_{1},r_{2}\in \mathbb{R}^{n}.$

2) (homogeneous) $l(ar)=|a|l(r)$ for any $a\in \mathbb{R}$ and $r\in \mathbb{%
R}^{n}.$

3) (conjugation invariant) $l(r)=l(Ar)$ for any $r\in \mathbb{R}^{n}.$
\end{lemma}

\begin{proof}
Define $l_{\mathbb{Z}^{n}}=l_{\mathbb{Z}^{n}}^{\prime }.$ For any $%
q=(x_{1},\cdots ,x_{n})^{T}\in \mathbb{Q}^{n}$, let $d$ be the least common
multiple of the denominators of $x_{i}.$ Define $l(q)=\frac{1}{d}l(dq).$ For
any integer $k>0,$ we have $l(kq)=\frac{1}{|k|d}l(dkq)$ and $l(q)=\frac{1}{k}%
l(kq).$ Therefore, $l(rq)=|r|l(q)$ for any rational number $r$ and any
element $q\in \mathbb{Q}^{n}.$ Suppose that $\{q_{i}\}\subset \mathbb{Q}^{n}$
is a Cauchy sequence. For any $\varepsilon >0,$ there exists $N$ such that
when $k,m>N,$ we have%
\begin{equation*}
\Vert q_{k}-q_{m}\parallel <\varepsilon .
\end{equation*}%
Suppose that the $i$-th component $(q_{k})_{i}=\frac{r_{ki}}{s_{ki}}$ for
coprime integers $r_{ki},s_{ki}.$ Denote by $\{e_{1},...,e_{n}\}$ the
standard basis of $\mathbb{Z}^{n}.$ Then 
\begin{eqnarray*}
|l(q_{k})-l(q_{m})| &=&|\frac{1}{{lcm}(s_{k1},s_{k2},...,s_{kn})}l(%
{lcm}(s_{k1},s_{k2},...,s_{kn})(\frac{r_{k1}}{s_{k1}},\frac{r_{k2}}{%
s_{k2}},...,\frac{r_{kn}}{s_{kn}}))-l(q_{m})| \\
&\leq &\sum_{i=1}^{n}l((\frac{r_{ki}}{s_{ki}}-\frac{r_{mi}}{s_{mi}})e_{i}) \\
&\leq &\sum_{i=1}^{n}|\frac{r_{ki}}{s_{ki}}-\frac{r_{mi}}{s_{mi}}|l(e_{i}) \\
&<&n\varepsilon Max\{l(e_{1}),...,l(e_{n})\}.
\end{eqnarray*}%
This proves that $\{l(q_{i})\}$ is a Cauchy sequence. For any $r\in \mathbb{R%
}^{n},$ choose a rational sequence $q_{i}\rightarrow r$ and define $%
l(r)=\lim l(q_{i}).$ For any $r_{1},r_{2}\in \mathbb{R}^{n},$ let rational
sequences $q_{1i}=\frac{r_{1i}}{s_{1i}}\rightarrow r_{1},q_{2i}=\frac{r_{2i}%
}{s_{2i}}\rightarrow r_{2}$ (here $r_{1i},r_{2i}\in \mathbb{Z}%
^{n},s_{1i},s_{2i}\in \mathbb{Z}$). We have that 
\begin{eqnarray*}
l(q_{1i}+q_{2i}) &=&l(\frac{r_{1i}}{s_{1i}}+\frac{r_{2i}}{s_{2i}}) \\
&=&l(\frac{r_{1i}s_{2i}+r_{2i}s_{1i}}{s_{1i}s_{2i}})=\frac{1}{s_{1i}s_{2i}}%
l(r_{1i}s_{2i}+r_{2i}s_{1i}) \\
&\leq &\frac{1}{s_{1i}s_{2i}}(l(r_{1i}s_{2i})+l(r_{2i}s_{1i})) \\
&=&l(\frac{r_{1i}}{s_{1i}})+l(\frac{r_{2i}}{s_{2i}})=l(q_{1i})+l(q_{2i}).
\end{eqnarray*}%
Therefore,%
\begin{eqnarray*}
l(r_{1}+r_{2}) &=&\lim l(q_{1i}+q_{2i}) \\
&\leq &\lim l(q_{1i})+\lim l(q_{2i})=l(r_{1})+l(r_{2}).
\end{eqnarray*}%
This proves the subadditivity of $l$ on $\mathbb{R}^{n}.$ From the
definition, we obviously have 2).  Since the action of $A$ on $\mathbb{Z}^{n}$
is linear, the action extends obviously to $\mathbb{R}^{n}.$  Note that the action of $A$ on $\mathbb{Z}^2$ coincides with an action by conjugation in $G$. Since the length is invariant under conjugation, we have 3).
\end{proof}

\begin{lemma}
Any length function $l^{\prime }:\mathbb{Z}^{n}\rightarrow \mathbb{R}_{\geq 0}$
can be extended to be a unique continuous length function $l:\mathbb{R}%
^{n}\rightarrow \mathbb{R}_{\geq 0}.$
\end{lemma}

\begin{proof}
 For any $q=(x_{1},\cdots ,x_{n})^{T}\in \mathbb{Q}^{n}$, let $d$ be
the least common multiple of the denominators of $x_{i}.$ Define $l(q)=\frac{%
1}{d}l^{\prime }(dq).$ For any $r\in \mathbb{R}^{n},$ choose a rational
sequence $q_{i}\rightarrow r$ and define $l(r)=\lim l(q_{i}).$ The proof of Lemma \ref{key} checks that $l$ is a length function. The
uniqueness of $l$ comes from the fact that a continuous function on $\mathbb{%
R}^{n}$ depends only on its image on $\mathbb{Q}^{n}.$
\end{proof}

\begin{lemma}
\label{norm1}Let $A\in \mathrm{GL}_{n}(\mathbb{Z})$ be a $\mathbb{C}$%
-diagonalizable matrix without eigenvalues of norm $1.$ Then any length
function $l$ of $G=\mathbb{Z}^{n}\rtimes _{A}\mathbb{Z}$ vanishes on $%
\mathbb{Z}^{n}.$
\end{lemma}

\begin{proof}
Consider the real Jordan form 
\begin{equation*}
\begin{bmatrix}
A_{1} &  &  &  \\ 
& A_{2} &  &  \\ 
&  & \ddots &  \\ 
&  &  & A_{r}%
\end{bmatrix}%
\end{equation*}%
of $A,$ where $A_{i}$ is either a real Jordan block or 
\begin{equation*}
a%
\begin{bmatrix}
\cos \phi & -\sin \phi \\ 
\sin \phi & \cos \phi%
\end{bmatrix}%
,
\end{equation*}%
where $a\neq 1$ is a positive real number. Choose the corresponding basis $%
\{v_{1},..,v_{n}\}$ from corresponding $A_{i}$-invariant subspaces for $%
\mathbb{R}^{n}.$ Note that the length function $l$ can be extended to be a
function on $\mathbb{R}^{n}$ by Lemma \ref{key}. If $Av_{i}=\lambda
_{i}v_{i} $ for $|\lambda _{i}|\neq 1,$ then $l(v_{i})=|\lambda
_{i}|l(v_{i}) $ implies $l(v_{i})=0.$ Otherwise, 
\begin{equation*}
l(v_{i})=l(A^{k}v_{i})=a^{k}l((\frac{1}{a}A_{i})^{k}v_{i})
\end{equation*}%
for any integer $k.$ Since $\frac{1}{a}A_{i}$ is a rotation matrix and $l$
(a continuous function by Lemma \ref{key}) is bounded on compact set, we get
that $l(v_{i})=0$ when $a\neq 1.$ This shows that the length function $l$
vanishes on a basis of $\mathbb{R}^{n}$ and thus vanishes on the whole $%
\mathbb{R}^{n}$ by Lemma \ref{key}.
\end{proof}

\begin{lemma}
\label{discrete}Let $l:\mathbb{R}^{n}\rightarrow \mathbb{R}$ be a continuous length
function. If the restriction $l|_{\mathbb{Z}^{n}}$ is discretely purely
positive, then $l$ is purely positive.
\end{lemma}

\begin{proof}
By \cite{Tao}, there is a Banach space $(B,\parallel \parallel )$ and an
additive group homomorphism $f:\mathbb{R}^{n}\rightarrow B$ such that $%
l(x)=\parallel f(x)\parallel $ for any $x\in \mathbb{R}^{n}.$ If $l(x)=0$
for some $x\neq 0,$ the kernel $\ker f$ is non-trivial and ${Im}f=%
\mathbb{R}^{m},m<n.$ Since $l|_{\mathbb{Z}^{n}}$ is discretely purely
positive, the group $\mathbb{Z}^{n}$ acts freely and properly
discontinuously on ${Im}f=\mathbb{R}^{m}.$ This implies $m=n,$ a
contradiction. Therefore, $l$ is purely positive.
\end{proof}

\begin{remark}
Without the condition that $l|_{\mathbb{Z}^{n}}$ is discrete, Lemma \ref%
{discrete} is not true. For example, let $v_{1}=(1,\sqrt{2}),v_{2}=(-\sqrt{2}%
,1)\in \mathbb{R}^{2}.$ For the decomposition $\mathbb{R}^{2}=\mathbb{R}%
v_{1}\bigoplus \mathbb{R}v_{2},$ let $f:\mathbb{R}^{2}\rightarrow \mathbb{R}$
be the projection onto the first component. The length function $l$ defined
by $l(x)=|f(x)|$ is purely positive on $\mathbb{Z}^{2},$ since the line $%
\mathbb{R}v_{1}$ has irrational slope.
\end{remark}

\begin{lemma}
\label{finite}Let $A\in \mathrm{GL}_{n}(\mathbb{Z}).$ If the group $G=%
\mathbb{Z}^{n}\rtimes _{A}\mathbb{Z}$ admits a discrete purely positive
length function function $l,$ then $G$ is virtually abelian.
\end{lemma}

\begin{proof}
By Lemma \ref{key}, the restriction of $l$ to $\mathbb Z^n$ extends to a continuous length function \[ l\colon\mathbb R^n\longrightarrow\mathbb R_{\geq0} \] satisfying \[ l(rx)=|r|l(x) \qquad\text{and}\qquad l(Ax)=l(x) \] for all $r\in\mathbb R$ and $x\in\mathbb R^n$. Since $l|_{\mathbb Z^n}$ is discretely purely positive, Lemma~2.4 implies that $l(x)>0$ for every $0\neq x\in\mathbb R^n$. Thus $l$ is a norm on $\mathbb R^n$. The equality \[ l(Ax)=l(x) \] shows that $A$ is a linear isometry of the finite-dimensional normed vector space $(\mathbb R^n,l)$. Hence the set \[ \{A^k:k\in\mathbb Z\} \] is bounded in operator norm. It follows that every complex eigenvalue of $A$ has modulus one. It also follows that $A$ is semisimple over $\mathbb C$, since a nontrivial Jordan block corresponding to an eigenvalue of modulus one would make the norms of the powers of $A$ grow polynomially. Let $\lambda$ be an eigenvalue of $A$. Since $A$ has integral entries, $\lambda$ is an algebraic integer. Its minimal polynomial divides the characteristic polynomial of $A$. Therefore every algebraic conjugate of $\lambda$ is also an eigenvalue of $A$ and consequently has modulus one. Kronecker's theorem implies that $\lambda$ is a root of unity. There are only finitely many eigenvalues of $A$. Hence there exists a positive integer $q$ such that $\lambda^q=1$ for every eigenvalue $\lambda$ of $A$. Since $A$ is semisimple, this gives  $A^q=I_n.$ Therefore \[ \mathbb Z^n\rtimes_A q\mathbb Z \cong \mathbb Z^n\times\mathbb Z \] is a finite-index subgroup of $\mathbb Z^n\rtimes_A\mathbb Z.$ Thus $G$ is virtually abelian.
\end{proof}

On the other hand, Conner \cite{conner} (Example 7.1) gave an example $%
\mathbb{Z}^{n}\rtimes _{A}\mathbb{Z}$ with purely positive word length,
where 
\begin{equation*}
A=%
\begin{bmatrix}
0 & 0 & 0 & -1 \\ 
1 & 0 & 0 & 2 \\ 
0 & 1 & 0 & -1 \\ 
0 & 0 & 1 & 2%
\end{bmatrix}%
\end{equation*}%
is an irreducible matrix with an eigenvalue of norm 1. Recall that an
integral matrix $A_{n\times n}$ is irreducible if it has no proper
non-trivial invariant subspace in $\mathbb{Q}^{n}.$

\bigskip Let $\mathbb{C}^{n}>\mathbb{Z}^{n}$ be a normed vector space, given
by 
\begin{equation*}
\Vert a_{1}e_{1}+a_{2}e_{2}+\cdots +a_{n}e_{n}\Vert =\sqrt{a_{1}\bar{a}%
_{1}+a_{2}\bar{a}_{2}+\cdots +a_{n}\bar{a}_{n}},
\end{equation*}%
where each $a_{i}\in \mathbb{C}$ and $\{e_{1},e_{2},...,e_{n}\}$ is the
standard basis. For each 
\begin{equation*}
g=(z,t^{i})\in \mathbb{C}^{n}\rtimes _{A}\mathbb{Z},
\end{equation*}%
where $t$ is a generator of $\mathbb{Z}$, let 
\begin{equation*}
L(g)=\inf \{\sum_{j=1}^{k}(\Vert z_{j}\Vert
+|i_{j}|):g=(z_{1},t^{0})(0,t^{i_{1}})(z_{2},t^{0})(0,t^{i_{2}})\cdots
(z_{k},t^{0})(0,t^{i_{k}})
\end{equation*}%
for some $z_{1},z_{2},...,z_{k}\in \mathbb{C}^{n}$ and integers $%
i_{1},i_{2},...,i_{k}\}.$ From the definition, we have 
\begin{equation*}
L(gh)\leq L(g)+L(h)
\end{equation*}%
and $L(g)=L(g^{-1})$ for any $g,h\in \mathbb{C}^{n}\rtimes _{A}\mathbb{Z}$.
Define 
\begin{equation*}
l(g)=\lim_{n\rightarrow \infty }\frac{L(g^{n})}{n},
\end{equation*}%
called the stable word length for $g\in \mathbb{Z}^{n}\rtimes _{A}\mathbb{Z}%
. $

Using the same idea as Conner \cite{conner}, we can prove the following
result.

\begin{lemma}
\label{stableword}Let $A\in \mathrm{GL}_{n}(\mathbb{Z})$ be an irreducible $%
\mathbb{C}$-diagonalizable matrix with a norm-one eigenvalue. Then the
stable word length on the semi-direct product $\mathbb{Z}^{n}\rtimes _{A}%
\mathbb{Z}$ is purely positive.
\end{lemma}

\begin{proof}
Let $v\in \mathbb{C}^{n}$ be a unit eigenvector such that $Av=\lambda v$ for
an eigenvalue $\lambda $ with $\Vert \lambda \Vert =1.$ Since $A$ is
diagonalizable over $\mathbb{C}$, the space $\mathbb{C}^{n}=\bigoplus P_{i}$
is a direct sum of eigenspaces. In the eigenspace $P_{1}$ corresponding to
the eigenvalue $\lambda ,$ choose $\{v,v_{1},..,v_{k}\}$ as a basis for $%
P_{1}.$ Denote by $(\mathbb{C}v)^{\perp }$ the $\mathbb{C}$-subspace spanned
by $\{v_{1},..,v_{k}\}$ and other eigenspaces. By construction, we have $%
\mathbb{C}^{n}=\mathbb{C}v\bigoplus (\mathbb{C}v)^{\perp }$ and $(\mathbb{C}%
v)^{\perp }$ is $A$-invariant. Let $p:\mathbb{C}^{n}\rightarrow \mathbb{C}v$
be the projection. Note that $(p\circ A)(w)=\lambda p(w)$ for any $w\in 
\mathbb{C}^{n}.$

Let $\{e_{i}\}_{i=1}^{n}$ be the standard basis of $\mathbb{Z}^{n}$ and $%
S=\{e_{1},...,e_{n},t\}$ a generating set of $\mathbb{Z}^{n}\rtimes _{A}%
\mathbb{Z}$. For arbitrary $z\in \mathbb{Z}^{n},$ write $%
(z,1)=(z_{1},t^{i_{1}})(z_{2},t^{i_{2}})\cdots (z_{k},t^{i_{k}}),$ for some $%
z_{i}=\sum_{s=1}^{k_{i}}a_{i_{s}}e_{i_{s}}\in \mathbb{Z}^{n},a_{i_{s}},i_{j}%
\in \mathbb{Z}$ such that where $L(z)=\sum (|a_{i_{j}}|+|i_{j}|)\ $is the
word length. We have 
\begin{eqnarray*}
z &=&(z,1)=(z_{1},t^{i_{1}})(z_{2},t^{i_{2}})\cdots (z_{k},t^{i_{k}}) \\
&=&(z_{1}+A^{i_{1}}z_{2},t^{i_{1}+i_{2}})\cdots (z_{k},t^{i_{k}}) \\
&=&z_{1}+A^{i_{1}}z_{2}+\cdots +A^{i_{1}+\cdots +i_{k-1}}z_{k}.
\end{eqnarray*}%
Applying the projection $p$ to the previous identity, we have 
\begin{eqnarray*}
p(z) &=&p(z_{1})+\lambda ^{i_{1}}p(z_{2})+\cdots +\lambda
^{i_{1}+\cdots +i_{k-1}}p(z_{k}), \\
\Vert p(z)\Vert &=&\Vert p(z_{1})+\lambda ^{i_{1}}p(z_{2})+\cdots
+\lambda ^{i_{1}+\cdots +i_{k-1}}p(z_{k})\Vert \\
&=&\Vert \sum_{l,s}a_{l_{s}}\lambda ^{i_{1}+\cdots +i_{l-1}}p(e_{l_{s}})\Vert
\\
&\leq &\sum_{l,s}|a_{l_{s}}|\cdot \Vert p(e_{l_{s}})\Vert \leq cL(z)
\end{eqnarray*}%
for $c=\max_{1\leq i\leq n}\Vert p(e_{i})\Vert .$ We use the fact that $\lambda$ has norm one in the last inequality. Since $\Vert p(z^{n})\Vert
=|n|\Vert p(z)\Vert \leq cL(z^{n})$ for each integer $n,$ we have%
\begin{equation*}
l(z)=\lim \frac{L(z^{n})}{n}\geq \frac{1}{1+c} \Vert p(z)\Vert .
\end{equation*}

Let%
\begin{equation*}
I=\{g\in \mathbb{Z}^{n}\rtimes _{A}\mathbb{Z}:l(g)=0\}.
\end{equation*}%
For any $g=(z,t^{i}),$ we have $L(g)\geq |i|.$ This implies that $I\subset 
\mathbb{Z}^{n}.$ Since $l$ is subadditive and homogeneous on $\mathbb{Z}%
^{n}, $ we have 
\begin{eqnarray*}
l(x-y) &\leq &l(x)+l(y)=0, \\
l(nx) &=&nl(x)=0
\end{eqnarray*}%
for any $x,y\in I,$ implying that the set $I$ is a subgroup. Furthermore, as 
$l$ is conjugate invariant and homogeneous, $I$ is actually an $A$-invariant
direct summand of $\mathbb{Z}^{n}$. Since $A$ is irreducible, we have either 
$I=\mathbb{Z}^{n}$ or $I=0.$ But if $I=\mathbb{Z}^{n},$ we will have $\Vert
p(z)\Vert =0$ for each $z$ since $l(z)\geq \frac{1}{1+c} \Vert p(z)\Vert $. This is
impossible, since $\mathbb{Z}^{n}$ spans $\mathbb{C}^{n}$. Therefore, we
have that $I=0$ and $l$ is purely positive.
\end{proof}

Combining Lemma \ref{norm1} and Lemma \ref{stableword}, we have the
following.

\begin{corollary}
Let $A\in \mathrm{GL}_{n}(\mathbb{Z})$ be an irreducible $\mathbb{C}$%
-diagonalizable matrix. The semi-direct product $\mathbb{Z}^{n}\rtimes _{A}%
\mathbb{Z}$ is purely positive if and only if there is a norm-one eigenvalue
of $A.$
\end{corollary}

\section{Proofs of theorems}

\begin{lemma}
\label{nil}(\cite{ye}, Corollary 4.7) A purely positive finitely generated
nilpotent group is virtually abelian.
\end{lemma}

\begin{remark}
\label{remk}Without the condition of finite generation, Lemma \ref{nil} is
not true. For each integer $n\geq 2,$ let $H(\mathbb{Z}/n)$ the group
consisting of all $3\times 3$ strictly upper triangular matrices with
entries in $\mathbb{Z}/n.$ Let $G=(\bigoplus_{n\geq 2}H(\mathbb{Z}/n))\times 
\mathbb{Z}$. The group $G$ is nilpotent, with a purely positive length
function coming from the projection $G\rightarrow \mathbb{Z}$. However, $G$
is not virtually abelian. Note that the statement of Corollary 6.1 in \cite%
{conner} is not precise. The correct reference there should be Corollary
3.6.2 of \cite{concom}.
\end{remark}

\begin{lemma}
\label{6.2}Let $G$ be a solvable group equipped with a purely positive
length function $l.$ Let $A$ be a maximal abelian normal subgroup of $G$.
Assume that $A$ is torsion-free. Then $C_{G}(A)=A$, where $C_{G}(A)$ is the
centralizer of $A$ in $G.$
\end{lemma}

\begin{proof}
Note that $C_{G}(A)$ is a normal subgroup of $G$. Actually, for any $x\in
C_{G}(A),t\in G,a\in A,$ we have $%
txt^{-1}a=txt^{-1}a(tt^{-1})=tx(t^{-1}at)t^{-1}=t(t^{-1}at)xt^{-1}=atxt^{-1}. 
$ Suppose that $C_{G}(A)\neq A.$ Let $Q$ be a maximal abelian characteristic
subgroup of the solvable group $C_{G}(A)/A$. Note that $Q$ is non-trivial,
as it can contain the last non-trivial term in the derived series of $%
C_{G}(A)/A.$ Since $C_{G}(A)/A$ is normal in $G/A,$ the group $Q$ is normal
in $G/A.$ Let $B<G$ be the normal subgroup corresponding to $Q,$ i.e. $Q$ is
the preimage of $Q$ in $G.$ Since the commutator subgroup $1\neq \lbrack
B,B]<A,$ there exist elements $b_{1},b_{2}\in B$ such that the commutator $%
1\neq \lbrack b_{1},b_{2}]\in A.$ But the subgroup $\langle
b_{1},b_{2}\rangle <B$ is a quotient of the Heisenberg group $\langle x,y,z: xz=zx,yz=zy,z=[x,y] \rangle$ by mapping $x \rightarrow b_1, y\rightarrow b_2$. Lemma \ref{nil} implies that
any length function of $\langle b_{1},b_{2}\rangle $ vanishes on the center $%
[b_{1},b_{2}]$, which is a contradiction.
\end{proof}

\begin{theorem}
\label{th0}Let $G$ be a solvable group with a purely positive length
function $l.$ Suppose that $G$ has a finite virtual cohomological
dimension. There exists a finite-index subgroup $K<G$ such that

1) either $K$ is abelian, or

2) there exits abelian groups $A,B$ fitting into an exact sequence%
\begin{equation*}
1\rightarrow A\rightarrow K\rightarrow B\rightarrow 1
\end{equation*}%
satisfying that the action of $B$ on $A\bigotimes_{\mathbb{Z}}\mathbb{Q}$ is
effective (i.e. faithful) and each non-identity element of $B$ is
non-unipotent in $\mathrm{Aut}(A\bigotimes_{\mathbb{Z}}\mathbb{%
Q)\hookrightarrow }\mathrm{GL}_{n}(\mathbb{C}).$ In particular, $G$ is
either virtually abelian or virtually non-nilpotent abelian-by-abelian.
\end{theorem}

For examples satisfying 2), see Lemma \ref{stableword}.

\begin{proof}
Let $H<G$ be a finite-index subgroup of finite cohomological dimension. Let $%
A$ be a maximal abelian normal subgroup of $H.$ Since $H$ is torsion-free, $%
A $ is torsion-free. The quotient group $H/A$ acts on $A$ by conjugation. By
Lemma \ref{6.2}, the action is effective. Note that $A$ embeds into $%
A\bigotimes_{\mathbb{Z}}\mathbb{Q\cong Q}^{n}$ for some integer $n.$ The
group $H/A$ is isomorphic to a subgroup of $\mathrm{GL}_{n}(\mathbb{Q}%
)\hookrightarrow \mathrm{GL}_{n}(\mathbb{C}),$ by considering its conjugate
action on $A.$ Therefore, $H/A$ has a finite-index subgroup $K/A$ (for some $%
K<G)$, which is conjugate to a subgroup of the upper triangular matrices in $%
\mathrm{GL}_{n}(\mathbb{C})$ (cf. \cite{s}, Theorem 3, page 33). 

Let $\overline g\in K/A$ be an element whose induced action \[ \varphi\colon A\otimes_{\mathbb Z}\mathbb Q \longrightarrow A\otimes_{\mathbb Z}\mathbb Q \] is unipotent, and choose a lift $g\in K$ of $\overline g$. We claim that $\overline g=1$. Suppose first that $\varphi\neq I$, and put \[ N=\varphi-I. \] The endomorphism $N$ is nonzero and nilpotent. Since $A$ spans $A\otimes_{\mathbb Z}\mathbb Q$, there exist $a\in A$ and an integer $r\geq1$ such that \[ N^r a\neq0 \qquad\text{and}\qquad N^{r+1}a=0. \] Put \[ b=N^{r-1}a \qquad\text{and}\qquad c=N^r a. \] Then $b,c\in A$, $c\neq0$, and \[ \varphi(b)=b+c, \qquad \varphi(c)=c. \] Using additive notation in $A$, it follows that \[ g^m b g^{-m}=b+mc \] for every positive integer $m$. Hence conjugacy invariance and subadditivity on the abelian group $A$ give \[ \begin{aligned} m l(c) &= l(mc)\\ &\leq l(b+mc)+l(-b)\\ &= 2l(b). \end{aligned} \] Letting $m\to\infty$ gives $l(c)=0$, contradicting pure positivity. Therefore $\varphi=I$. Since the action of $H/A$ on $A\otimes_{\mathbb Z}\mathbb Q$ is faithful, we conclude that $\overline g=1.$ Thus $K/A$ contains no nonidentity unipotent element. The commutator subgroup of a group of upper triangular matrices consists of unipotent matrices. Hence \[ [K/A,K/A]=1, \] and $K/A$ is abelian. If $K/A$ is finite, then $A$ has finite index in $H$, and the first alternative of the theorem holds after replacing $K$ by $A$. If $K/A$ is infinite, put $ B=K/A.$ This gives the required exact sequence \[ 1\longrightarrow A\longrightarrow K\longrightarrow B\longrightarrow1. \] Suppose that $K$ is nilpotent of class $c$.  For $g\in K$, let $\varphi_g$ be conjugation by $g$ on $A\otimes_{\mathbb{Z}}\mathbb{Q}$.  In additive notation, $(\varphi_g-I)^j(a)$ is represented by a $j$-fold iterated commutator of $a$ with $g$.  Hence
\[
(\varphi_g-I)^c=0.
\]
Thus every element of $K/A$ acts unipotently.  The preceding argument shows that the identity is the only unipotent element of $K/A$, so $K/A=1$.  Consequently, whenever $K/A$ is nontrivial (and in particular in the second alternative), $K$ is not nilpotent. 

\end{proof}

\bigskip

\begin{proof}[Proof of Theorem \protect\ref{th1}]
Let $K$ be a finite-index torsion-free subgroup of $G.$ Let $A$ be a maximal
normal abelian group of $K.$ Since $K/A$ acts effectively on $A$ by Lemma %
\ref{6.2}, we know that $K/A$ is isomorphic to a subgroup of $\mathrm{GL}%
_{n}(\mathbb{Q})\hookrightarrow \mathrm{GL}_{n}(\mathbb{C}).$ There is a
finite-index subgroup $K_{1}/A$ of $K/A,$ which is upper triangularizable.

The argument in the proof of Theorem \ref{th0} shows that every unipotent element of $K_1/A$ is the identity. Since the commutator subgroup of a group of upper triangular matrices consists of unipotent matrices, we have \[ [K_1/A,K_1/A]=1. \] Therefore $K_1/A$ is abelian and $K_1$ is metabelian.

Put \[ M=[K_1,K_1] \qquad\text{and}\qquad Q=(K_1)_{\mathrm{ab}}. \] The Kaloujnine--Krasner embedding theorem (cf. \cite{lr}, 11.3.3, page 252) applied to the exact sequence \[ 1\longrightarrow M\longrightarrow K_1\longrightarrow Q\longrightarrow1 \] gives an embedding \[ K_1\hookrightarrow M^Q\rtimes Q, \] where \[ M^Q=\prod_{q\in Q}M \] is the unrestricted direct product and $Q$ acts by the regular shift action. Since $M$ is abelian, $M^Q$ is naturally a $\mathbb ZQ$-module. Choose a finite generating set \[ k_1,\ldots,k_d \] of $K_1$, and write the image of $k_i$ as $(f_i,q_i)$. Let \[ H\leq M^Q \] be the $\mathbb ZQ$-submodule generated by $f_1,\ldots,f_d$. The multiplication and inverse formulas in $M^Q\rtimes Q$ show that the base component of every word in the $k_i^{\pm1}$ belongs to $H$. Therefore \[ K_1\hookrightarrow H\rtimes Q. \] The finitely generated abelian group $Q$ contains a finite-index subgroup $Q_0\cong\mathbb Z^n$ for some $n\geq0$. Let $K_2$ be the full preimage of $Q_0$ in $K_1$. Then $K_2$ has finite index in $K_1$ and \[ K_2\hookrightarrow H\rtimes Q_0 \cong H\rtimes\mathbb Z^n. \] Since $Q_0$ has finite index in $Q$, the group ring $\mathbb ZQ$ is a finitely generated module over $\mathbb ZQ_0$. Hence $H$ is a finitely generated $\mathbb ZQ_0\cong\mathbb Z[\mathbb Z^n]$ module.

The group $K_1$ is finitely generated and metabelian, and its derived subgroup is torsion-free because $K_1$ is torsion-free. The Levi\v{c}--Remeslennikov linearity theorem (see Levi\v{c} \cite{le} and
Remeslennikov \cite{Re}) for finitely generated metabelian groups with torsion-free derived subgroup therefore gives a faithful finite-dimensional representation of $K_1$ over a field of characteristic zero. Replacing that field by the subfield generated over $\mathbb Q$ by the matrix entries of the images of a finite generating set, we may assume that the field is finitely generated over $\mathbb Q$. Such a field embeds into $\mathbb C$. Thus $K_1$, and hence $K_2$, is $\mathbb C$-linear. If $G$ is not virtually abelian, then $K_2$ is not nilpotent. Otherwise Lemma~3.1 would imply that $K_2$ is virtually abelian, and hence that $G$ is virtually abelian.
\end{proof}

\begin{theorem}
\label{th2}Let $G$ be a solvable group with a purely positive length
function. Suppose that $G$ is virtually torsion-free and any abelian
subgroup $A$ is finitely generated. Then $G$ is virtually either $\mathbb{Z}%
^{n}$ or a non-nilpotent group $B$ of the form%
\begin{equation*}
1\rightarrow \mathbb{Z}^{m}\rightarrow B\rightarrow \mathbb{Z}%
^{k}\rightarrow 1
\end{equation*}%
for some integers $n,m,k.$ Here $\mathbb{Z}^{k}$ acts on $\mathbb{Z}^{m}$
effectively by non-unipotent matrices.
\end{theorem}

\begin{proof}[Proof of Theorem \protect\ref{th2}]
We use the notation as in the proof of Theorem \ref{th1}. By the assumption, 
$A=\mathbb{Z}^{m}$ for some $n.$ Since the group $K_{1}/A<\mathrm{GL}_{n}(%
\mathbb{Z})$ is solvable, it is poly-cyclic. In particular, $K_{1}/A$ is
finitely generated abelian group and there is a finite-index subgroup $%
\mathbb{Z}^{k}$ for some $k.$ Let $B<K_{1}$ be the preimage of $\mathbb{Z}%
^{k}<K_{1}/A$ to finish the proof. 
The argument in the proof of Theorem \ref{th0} shows that every non-identity element of $\mathbb Z^k$ acts on $\mathbb Z^m$ by a non-unipotent invertible matrix. If $B$ is virtually abelian, this gives the former case. If $B$ is not virtually abelian, then $B$ cannot be nilpotent by Lemma~3.1. This gives the latter case.
\end{proof}

\bigskip

\begin{proof}[Proof of Theorem \protect\ref{th3}]

Pass to a torsion-free finite-index subgroup of $G$, and continue to denote this subgroup by $G$. If $G=1$, there is nothing to prove. Let $A$ be a maximal abelian normal subgroup of $G$. Since $G$ is nontrivial and solvable, $A$ is nontrivial. Put $d=\operatorname{cd}_{\mathbb Z}(G).$  We first show that $A$ has finite rational rank. If $a_1,\ldots,a_m\in A$ are linearly independent, then \[ \langle a_1,\ldots,a_m\rangle\cong\mathbb Z^m. \] By monotonicity of cohomological dimension under passage to subgroups, \[ m = \operatorname{cd}_{\mathbb Z}(\mathbb Z^m) \leq \operatorname{cd}_{\mathbb Z}(G) = d. \] Therefore $ r:= \dim_{\mathbb Q} \bigl(A\otimes_{\mathbb Z}\mathbb Q\bigr) <\infty.$ Let $ V=A\otimes_{\mathbb Z}\mathbb Q.$ For $a\in A$ and a positive integer $n$, define \[ \overline l(a/n)=\frac{l(a)}{n}. \] This is well-defined. Indeed, if $a/n=b/m$ in $V$, then \[ ma=nb, \] and homogeneity gives \[ m l(a)=l(ma)=l(nb)=n l(b). \] The function $\overline l$ is subadditive and homogeneous over $\mathbb Q$. Choose a $\mathbb Q$-basis \[ a_1,\ldots,a_r \] of $V$ consisting of elements of $A$, and identify $V$ with $\mathbb Q^r$ using this basis. If \[ v=\sum_{i=1}^r q_i a_i, \qquad q_i\in\mathbb Q, \] then \[ \overline l(v) \leq \sum_{i=1}^r |q_i|l(a_i) \leq C\sum_{i=1}^r|q_i|, \] where \[ C=\max_{1\leq i\leq r}l(a_i)>0. \] Since $l$ is discrete and purely positive, there exists $\varepsilon>0$ such that \[ l(a)\geq\varepsilon \] for every $0\neq a\in A$. Hence, if \[ 0\neq a=\sum_{i=1}^r q_i a_i\in A, \] then \[ \varepsilon \leq l(a) \leq C\sum_{i=1}^r|q_i|. \] Thus \[ \sum_{i=1}^r|q_i| \geq \frac{\varepsilon}{C}. \] Applying this estimate to the difference of two distinct elements of $A$ shows that the coordinate image of $A$ is a discrete subgroup of $\mathbb R^r$. Every discrete subgroup of a finite-dimensional real vector space is finitely generated and free abelian. Since $A$ contains the basis elements $a_1,\ldots,a_r$, it has rank $r$. Therefore \[ A\cong\mathbb Z^r. \] By Lemma~3.3, \[ C_G(A)=A. \] Consequently, conjugation induces an injective homomorphism \[ G/A \hookrightarrow \operatorname{Aut}(A) \cong \operatorname{GL}_r(\mathbb Z). \] We claim that $G/A$ is a torsion group. Suppose that \[ \overline g=gA\in G/A \] has infinite order, and choose a lift $g\in G$. Then \[ \langle g\rangle\cap A=\{1\}, \] because $g^m\in A$ would imply $\overline g^{\,m}=1$. It follows that \[ \langle A,g\rangle \cong A\rtimes_{\varphi}\mathbb Z, \] where $\varphi$ is the automorphism of $A$ induced by conjugation by $g$. The restriction of $l$ to $\langle A,g\rangle$ is again discrete and purely positive. Lemma~2.6 therefore implies that $\varphi$ has finite order. Thus there exists $q>0$ such that \[ \varphi^q=I. \] Equivalently, $g^q$ centralizes $A$. Lemma~3.3 gives \[ g^q\in C_G(A)=A, \] and hence \[ \overline g^{\,q}=1, \] contradicting the assumption that $\overline g$ has infinite order. Therefore $G/A$ is torsion. Every solvable subgroup of $\operatorname{GL}_r(\mathbb Z)$ is polycyclic. Hence $G/A$ is a torsion polycyclic group and is therefore finite. Thus $A\cong\mathbb Z^r$ has finite index in $G$. Returning to the original group before passage to a finite-index subgroup, the intersection of the finitely many conjugates of $A$ is a normal, finite-index, finitely generated free abelian subgroup. Therefore the original group is a finite extension of $\mathbb Z^r$.
\end{proof}

\bigskip

\begin{proof}[Proof of Theorem \protect\ref{th-1}]
Let $\mathrm{Diff}^{\infty }(M)$ be the group of diffeomorphisms of $M.$ For
any $f,g\in \mathrm{Diff}^{\infty }(M)$ and integer $n,$ it is well-known
that the topological entropy satisfies $h_{\mathrm{top}}(f^{n})=|n|h_{%
\mathrm{top}}(f)$ and $h_{\mathrm{top}}(f)=h_{\mathrm{top}}(gfg^{-1})$ (cf. 
\cite{kat}, Cor. 3.14 and Prop. 3.1.7, page 111). Hu \cite{hu} (Theorem C)
proves that $h_{\mathrm{top}}(fg)\leq h_{\mathrm{top}}(f)+h_{\mathrm{top}%
}(g) $ when $fg=gf.$ This shows that the topological entropy $h_{\mathrm{top}%
}$ is a length function on the group $\mathrm{Diff}^{\infty }(M).$ By
Theorem \ref{th0}, the group $H$ is virtually either abelian or
meta-abelian. In any case, $H$ is virtually abelian-by-abelian. If there is
a uniform lower bound of the topological entropies, the length function is
discrete. Theorem \ref{th3} implies that $H$ is virtually a finitely
generated abelian group $\mathbb{Z}^{n}.$

For the Lyapunov-exponent alternative, define \[ l_D(f) = \max\left\{ \lim_{n\to\infty} \frac{1}{n} \log\left(\sup_{y\in M}\|D_yf^n\|\right), \, \lim_{n\to\infty} \frac{1}{n} \log\left(\sup_{y\in M}\|D_yf^{-n}\|\right) \right\}. \] The limits exist by subadditivity, and $l_D$ is a length function on $\operatorname{Diff}^{\infty}(M)$; see \cite[Lemma~2.5]{ye}. Suppose that \[ \chi(x,u;f) = \lim_{n\to\infty} \frac{1}{n} \log \frac{\|D_xf^n(u)\|}{\|u\|} >0. \] For every positive integer $n$, \[ \frac{\|D_xf^n(u)\|}{\|u\|} \leq \sup_{y\in M}\|D_yf^n\|. \] Consequently, \[ \chi(x,u;f) \leq \lim_{n\to\infty} \frac{1}{n} \log\left(\sup_{y\in M}\|D_yf^n\|\right) \leq l_D(f). \] Thus the Lyapunov-exponent hypothesis implies that $l_D$ is purely positive on $G$. A uniform positive lower bound for these exponents implies that $l_D$ is discrete. The conclusions now follow from Theorem~3.4 and Theorem~1.2, respectively.
\end{proof}

\bigskip

For groups of low cohomological dimensions, the solvable subgroup theorems
can be improved as the following. Note that this is not a contradiction to
Lemma \ref{stableword}, since a $2\times 2$ irreducible $\mathbb{C}$%
-diagonalizable invertible integer matrix cannot have an eigenvalue of norm $%
1.$

\begin{theorem}
\label{th4}Let $G$ be a solvable group with a purely positive length
function. Suppose that $G$ has virtual cohomological dimension $\mathrm{vcd}%
(G)\leq 3$ and any abelian subgroup $A$ is finitely generated. Then $G$ is a
finite extension of $\mathbb{Z}^{n}$ for some integer $n\leq 3.$
\end{theorem}

\begin{proof}[Proof of Theorem \protect\ref{th4}]
By Theorem \ref{th2}, the group $G$ is virtually either $\mathbb{Z}^{n}$ or
a non-nilpotent group $B$ of the form%
\begin{equation*}
1\rightarrow \mathbb{Z}^{m}\rightarrow B\rightarrow \mathbb{Z}%
^{k}\rightarrow 1
\end{equation*}%
for some non-negative integers $n,m,k.$ It is enough to rule out the second
case. Since cohomological dimension of $B$ is at most $3,$ we see that $%
m+k\leq 3.$ When $m=1$ or $m=3,$ the group $B$ is virtually abelian since $%
\mathbb{Z}^{k}$ acts effectively on $\mathbb{Z}^{m}$. When $m=2,$ the group $%
B=\mathbb{Z}^{2}\rtimes _{\phi }\mathbb{Z},$ a semi-direct product for some $%
\phi \in \mathrm{GL}_{2}(\mathbb{Z}).$ Consider the Jordan canonical form of 
$\phi .$ If all the eigenvalues $\lambda _{1},\lambda _{2}$ of $\phi $ have
norm different from $1,$ any length function will vanish on $\mathbb{Z}^{2}$
by Lemma \ref{norm1}. Otherwise, $\lambda _{1},\lambda _{2}=\pm 1$ or $%
\lambda _{1},\lambda _{2}=\bar{\lambda}_{1}$ are conjugate complex numbers
of norm $1.$ If $\phi $ is $\mathbb{C}$-diagonalizable, $\phi $ is of finite
order, in which case $B$ is virtually abelian. Otherwise, $\phi $ is
conjugate (in $\mathrm{GL}_{n}(\mathbb{C})$) to%
\begin{equation*}
\begin{bmatrix}
1 & 1 \\ 
0 & 1%
\end{bmatrix}%
\text{ or }%
\begin{bmatrix}
-1 & 1 \\ 
0 & -1%
\end{bmatrix}%
,
\end{equation*}%
in which case the matrix $\phi $ has the repeated eigenvalues $1$ or $-1.$
We can choose a basis $\{v_{1},v_{2}\}\subset \mathbb{Q}^{2}$ such that the
representation matrix of $\phi :\mathbb{Q}^{2}\rightarrow \mathbb{Q}^{2}$
with respect to the basis $\{v_{1},v_{2}\}$ is of the form%
\begin{equation*}
J=%
\begin{bmatrix}
1 & x \\ 
0 & 1%
\end{bmatrix}%
\text{ or }L=%
\begin{bmatrix}
-1 & y \\ 
0 & -1%
\end{bmatrix}%
\end{equation*}%
for some nonzero $x,y\in \mathbb{Q}.$ Note that a purely positive length
function on $\mathbb{Z}^{2}\rtimes _{\phi }\mathbb{Z}$ can be extended to be
a purely positive length function on $\mathbb{Q}^{2}$ by Lemma \ref{key}.
For any positive even integer $m,$ we have 
\begin{eqnarray*}
J^{m}v_{2} &=&mxv_{1}+v_{2},\text{ or} \\
L^{m}v_{2} &=&-myv_{1}+v_{2}
\end{eqnarray*}%
and 
\begin{eqnarray*}
m|x|l(v_{1}) &\leq &l(J^{m}v_{2})+l(v_{2})=2l(v_{2}),\text{ or} \\
m|y|l(v_{1}) &\leq &l(J^{m}v_{2})+l(v_{2})=2l(v_{2}).
\end{eqnarray*}%
Since $m$ is arbitrary, we get $l(v_{1})=0.$ This is a contradiction to the
fact that $l$ is purely positive. The proof is finished.
\end{proof}

\section{Stable norms}

Following \cite{conner, cf}, a real-valued function $L:G\rightarrow \lbrack
0,\infty )$ on a group $G$ is called a (semi-)norm if $L(gh)\leq
L(g)+L(h),L(g)=L(g^{-1})$ for any $g,h\in G.$ It's well-known that the
stable norm%
\begin{equation*}
\mathrm{sL}(g):=\lim_{n\rightarrow \infty }\frac{L(g^{n})}{n}
\end{equation*}%
is a length function (see \cite{ye}). In this section, we study stable norms
on the Heisenberg group 
\begin{eqnarray*}
H &=&\langle f,g,h\mid fh=hf,gh=hg,[f,g]=h\rangle \\
&=&\{%
\begin{bmatrix}
1 & a & c \\ 
& 1 & b \\ 
&  & 1%
\end{bmatrix}%
:a,b,c\in \mathbb{Z}\}.
\end{eqnarray*}%
Choose $S=\{f,g,f^{-1},g^{-1}\}$ to be a generating set and let $|w|_{S}$ be
the word length defined by $S$ for 
\begin{equation*}
w=%
\begin{bmatrix}
1 & a & c \\ 
& 1 & b \\ 
&  & 1%
\end{bmatrix}%
\in H.
\end{equation*}%
Denote $d(a,b,c)=|w|_{S},$ the word length. Blachere \cite{B} shows that $%
d(a,b,c)=d(-a,b,-c)=d(a,-b,-c)=d(-a,-b,c)=d(b,a,c)$ for any $a,b,c\in 
\mathbb{Z}$. The following is from Theorem 2.2 of \cite{B}.

\begin{lemma}
\label{nilword}Assume that $c\geq 0,a\geq 0,a\geq b\geq -a.$ When $b\geq
0,a\leq \sqrt{c},$ we have $d(a,b,c)=2\lceil 2\sqrt{c}\rceil -a-b.$ When $%
b\geq 0,a\geq \sqrt{c}$ and $ab>c,$ we have $d(a,b,c)=a+b.$
\end{lemma}

\begin{lemma}
\label{good}Let $L:H\rightarrow \lbrack 0,\infty )$ be a semi-norm on the
Heisenberg group $H.$ There are constants $K,C\geq 0$ such that%
\begin{equation*}
L(h^{n})\leq K\sqrt{n}+C
\end{equation*}%
for any integer $n\geq 0.$
\end{lemma}

\begin{proof}
Let $|w|_{S}$ be the word length of $w\in H$ with respect to the generating
set $S=\{f,g,f^{-1},g^{-1}\}.$ Suppose that $w=s_{1}s_{2}\cdots s_{|w|_{S}}$
with each $s_{i}\in S.$ It is clear that%
\begin{eqnarray*}
L(w) &\leq &L(s_{1})+L(s_{2})+\cdots +L(s_{|w|_{S}}) \\
&\leq &|w|_{S}\max \{L(s):s\in S\}.
\end{eqnarray*}%
Choose an integer $k$ such that $k^{2}\leq n<(k+1)^{2}.$ We have 
\begin{eqnarray*}
L(h^{n}) &\leq &|h^{n}|_{S}\max \{L(s):s\in S\} \\
&=&2\lceil 2\sqrt{n}\rceil \max \{L(s):s\in S\} \\
&\leq &K\sqrt{n}+C
\end{eqnarray*}%
for $K=4\max \{L(s):s\in S\},C=2\max \{L(s):s\in S\}+2.$ Note that the word
length $|h^{n}|_{S}=2\lceil 2\sqrt{n}\rceil ,$ twice the upper integer of $2%
\sqrt{n},$ by Lemma \ref{nilword}.
\end{proof}

\bigskip

\begin{proof}[Proof of Theorem \protect\ref{th5}]
For any smooth diffeomorphism $f_{1},$ we have that 
\begin{equation*}
\Vert Df_{1x}\Vert \leq \max \{\sup_{x\in M}\Vert Df_{1x}\Vert ,\sup_{x\in
M}\Vert Df_{1x}^{-1}\Vert \}.
\end{equation*}%
Let $L(f_{1})=\log \max \{\sup_{x\in M}\Vert Df_{1x}\Vert ,\sup_{x\in
M}\Vert Df_{1x}^{-1}\Vert \}.$ It's not hard to check that $%
L(f_{1}f_{2})\leq L(f_{1})+L(f_{2})$ and $L(f_{1})=L(f_{1}^{-1})$ for any
diffeomorphisms $f_{1},f_{2}.$ Lemma \ref{good} finishes the proof.
\end{proof}

\bigskip

\begin{corollary}
\label{swl}For any 
\begin{equation*}
w=%
\begin{bmatrix}
1 & a & c \\ 
& 1 & b \\ 
&  & 1%
\end{bmatrix}%
\in H,
\end{equation*}%
we have the stable word length 
\begin{equation*}
\mathrm{swl}_{S}(w)=\lim_{n\rightarrow \infty }\frac{|w^{n}|_{S}}{n}=|a|+|b|.
\end{equation*}
\end{corollary}

\begin{proof}
For any 
\begin{equation*}
w=%
\begin{bmatrix}
1 & a & c \\ 
& 1 & b \\ 
&  & 1%
\end{bmatrix}%
\in H,
\end{equation*}%
we have $\mathrm{swl}_{S}(w)=\mathrm{swl}_{S}(w_{0})$ for 
\begin{equation*}
w_{0}=%
\begin{bmatrix}
1 & a & 0 \\ 
& 1 & b \\ 
&  & 1%
\end{bmatrix}%
\end{equation*}%
since $\mathrm{swl}_{S}$ is a length function vanishing on the center of $H.$
For each positive integer $n,$ we have%
\begin{equation*}
w_{0}^{n}=%
\begin{bmatrix}
1 & na & \frac{n(n-1)ab}{2} \\ 
& 1 & nb \\ 
&  & 1%
\end{bmatrix}%
.
\end{equation*}%
After replacing $w_{0}$ by $w_{0}^{-1},$ we can always assume that $a\geq 0.$
When $a\geq b\geq 0,$ Lemma \ref{nilword} implies that 
\begin{eqnarray*}
|w_{0}^{n}|_{S} &=&na+nb, \\
\mathrm{swl}_{S}(w_{0}) &=&a+b.
\end{eqnarray*}%
When $b>a,$ we have 
\begin{eqnarray*}
|w_{0}^{n}|_{S} &=&d(na,nb,\frac{n(n-1)ab}{2})=d(nb,na,\frac{n(n-1)ab}{2}) \\
&=&nb+na, \\
\mathrm{swl}_{S}(w_{0}) &=&nb+na.
\end{eqnarray*}%
When $0\geq b>-a,$ we have 
\begin{eqnarray*}
|w_{0}^{n}|_{S} &=&d(na,nb,\frac{n(n-1)ab}{2})=d(na,n|b|,\frac{n(n-1)a|b|}{2}%
) \\
&=&na+n|b|.
\end{eqnarray*}%
When $b\leq -a,$ we have%
\begin{eqnarray*}
|w_{0}^{n}|_{S} &=&d(na,nb,\frac{n(n-1)ab}{2})=d(na,n|b|,\frac{n(n-1)a|b|}{2}%
) \\
&=&d(n|b|,na,\frac{n(n-1)a|b|}{2})=n|b|+na.
\end{eqnarray*}%
In any case, we have $\mathrm{swl}_{S}(w)=|a|+|b|.$
\end{proof}

It's not hard to see that the $\mathrm{swl}_{S}$ gives a norm on the
abelianization $H/\langle h\rangle $ which is generated by the images of $%
f,g.$

\begin{corollary}
Let $L:H\rightarrow \mathbb{R}$ be a semi-norm, i.e. $L(w_{1}w_{2})\leq
L(w_{1})+L(w_{2}),L(w_{1})=L(w_{1}^{-1})$ for any $w_{1},w_{2}\in H.$ There
exists constant $K$ such that 
\begin{equation*}
\mathrm{sL}(w)\leq K(|a|+|b|)
\end{equation*}%
for any 
\begin{equation*}
w=%
\begin{bmatrix}
1 & a & c \\ 
& 1 & b \\ 
&  & 1%
\end{bmatrix}%
\in H.
\end{equation*}
\end{corollary}

\begin{proof}
For each integer $k,$ write $w^{k}=s_{1}s_{2}...s_{n}$ for some $s_{i}\in
S=\{f,g,f^{-1},g^{-1}\}$ and $n=|w^{k}|_{S}.$ Then we have%
\begin{eqnarray*}
L(w^{k}) &\leq &|w^{k}|_{S}\max \{L(s):s\in S\}, \\
\mathrm{sL}(w) &=&\lim \frac{L(w^{k})}{k}\leq \mathrm{swl}(w)\max
\{L(s):s\in S\}.
\end{eqnarray*}%
Choose $K=\max \{L(s):s\in S\}$ and apply Corollary \ref{swl} to finish the
proof.
\end{proof}

\begin{corollary}
There exist length functions on $H$ which are not stable norms.
\end{corollary}

\begin{proof}

Let \[ \pi\colon H\longrightarrow H/[H,H]\cong\mathbb Z^2 \] be the abelianization homomorphism, so that \[ \pi\left( \begin{bmatrix} 1&a&c\\ 0&1&b\\ 0&0&1 \end{bmatrix} \right) = (a,b). \] Let \[ \mathcal P = \{(a,b)\in\mathbb Z^2:\gcd(a,b)=1\}/\{\pm1\} \] be the set of unoriented primitive directions in $\mathbb Z^2$. Define a function \[ \omega\colon\mathcal P\longrightarrow\mathbb R_{\geq0} \] by \[ \omega([(1,n)])=n^2 \qquad (n\geq1), \] and let $\omega$ vanish on every other primitive direction. We define \[ l\colon H\longrightarrow\mathbb R_{\geq0} \] as follows. If $\pi(w)=0$, put $l(w)=0$. If \[ \pi(w)=(a,b)\neq0, \] write \[ (a,b)=d(a_0,b_0), \qquad d=\gcd(|a|,|b|), \] where $(a_0,b_0)$ is primitive, and define \[ l(w)=d\,\omega([(a_0,b_0)]). \] The function $l$ is conjugacy invariant because conjugation acts trivially on the abelianization. It is homogeneous under powers by construction. Two elements $w_1,w_2\in H$ commute if and only if their images in $\mathbb Z^2$ are linearly dependent. Thus, if $w_1$ and $w_2$ commute, there exist a primitive vector $v\in\mathbb Z^2$ and integers $r,s$ such that \[ \pi(w_1)=rv, \qquad \pi(w_2)=sv. \] It follows that \[ \begin{aligned} l(w_1w_2) &= |r+s|\omega([v])\\ &\leq (|r|+|s|)\omega([v])\\ &= l(w_1)+l(w_2). \end{aligned} \] Therefore $l$ is a length function. For \[ g_{1,n} = \begin{bmatrix} 1&1&0\\ 0&1&n\\ 0&0&1 \end{bmatrix}, \] we have \[ \pi(g_{1,n})=(1,n), \] and hence \[ l(g_{1,n})=n^2. \] On the other hand, Corollary~4.4 shows that every stable norm $sL$ on $H$ satisfies \[ sL(g_{1,n})\leq K(1+n) \] for some constant $K$ independent of $n$. Thus $l$ has quadratic growth along the sequence $\{g_{1,n}\}$, whereas every stable norm has at most linear growth along this sequence. Therefore $l$ is not a stable norm.
\end{proof}

\section{Further applications}

In this section, we discuss the applications to various situations where
length functions are given explicitly. However, the results obtained are
usually narrow particular cases of known results (since the explicit length
functions generally have more properties). Our interest is to give a unified
view of these results.

\subsection{Group actions with positive topological entropy}

In this subsection, we collect several group actions with positive
topological entropies as motivations. The special linear group $\mathrm{SL}%
_{n}(\mathbb{Z})$ acts on the torus $T^{n}=\mathbb{R}^{n}/\mathbb{Z}^{n}$ by
quotient of the linear action on the real vector space $\mathbb{R}^{n}.$ For
an element $g\in \mathrm{SL}_{n}(\mathbb{Z}),$ it is well-known that the
topological entropy $h_{\mathrm{top}}(g)=\sum_{|\lambda |>1}\ln |\lambda |,$
where the sum is taken over all eigenvalues $\lambda $ of $g$ with absolute
value $|\lambda |>1$ (see \cite{akm}, page 316).

\begin{lemma}
Let $S<\mathrm{SL}_{n}(\mathbb{Z})$ be a solvable subgroup such that each
nontrivial element has a positive entropy. Then $S$ is virtually finitely
generated abelian.
\end{lemma}

\begin{proof}
Any solvable subgroup of $\mathrm{SL}_{n}(\mathbb{Z})$ is poly-cyclic (cf. 
\cite{s}, Corollary 1, page 26) and thus finitely generated. There is a
finite-index subgroup $H<S$ which is triangularizable in $\mathrm{GL}_{n}(%
\mathbb{C})$ (cf. \cite{s}, Theorem 3, page 33). Since each nontrivial
element $g$ of $H$ has a positive entropy $\sum_{|\lambda |>1}\ln |\lambda |$%
, the element $g$ is not unipotent. Therefore, the subgroup $H$ is
isomorphic to a subgroup of the diagonal subgroup of $\mathrm{GL}_{n}(%
\mathbb{C}).$ This proves that $H$ is finitely generated and abelian.
\end{proof}

Recall that a polynomial map $f:\mathbb{C}^{2}\rightarrow \mathbb{C}^{2}$
given by $f(x,y)=(f_{1}(x,y),f_{2}(x,y))$ is a diffeomorphism if there is a
polynomial inverse map $f^{-1}$. Let $\mathrm{Aut}(\mathbb{C}^{2})=\{f|f:%
\mathbb{C}^{2}\rightarrow \mathbb{C}^{2}$ is a polynomial homeomorphism$\}$
be the automorphism group of the plane $\mathbb{C}^{2}$. Let $A=\mathrm{GL}%
_{2}(\mathbb{C})\ltimes \mathbb{C}^{2}<\mathrm{Aut}(\mathbb{C}^{2})$ be the
group of affine transformations. For $\alpha ,\beta ,\gamma \in \mathbb{C}$
and a polynomial $p(y),$ let $f=f_{\alpha ,\beta ,\gamma ,p}:\mathbb{C}%
^{2}\rightarrow \mathbb{C}^{2}$ defined by%
\begin{equation*}
f_{\alpha ,\beta ,\gamma ,p}(x,y)=(\alpha x,\beta y)+(p(y),\gamma ).
\end{equation*}%
When $\alpha \beta \neq 0,$ the $f$ has the inverse $f^{-1}:\mathbb{C}%
^{2}\rightarrow \mathbb{C}^{2}$ defined by%
\begin{equation*}
f^{-1}(x,y) = \left( \alpha^{-1} \left( x-p\bigl(\beta^{-1}(y-\gamma)\bigr) \right), \, \beta^{-1}(y-\gamma) \right).
\end{equation*}%

Let $E=\{f_{\alpha ,\beta ,\gamma ,p}:\alpha \beta \neq 0\}<\mathrm{Aut}(%
\mathbb{C}^{2}).$ Then Jung-van der Kulk theorem (cf. \cite{es}, Section 5.1)
says that the group $\mathrm{Aut}(\mathbb{C}^{2})$ is isomorphic to the
amalgmated product $A\ast _{A\cap E}E.$ The topological entropy in the following corollary is the same as that discussed in Friedland-Milnor \cite{fm}.

\begin{corollary}
\label{last}A nontrivial solvable subgroup $S$ with finite virtual
cohomological dimension of $\mathrm{Aut}(\mathbb{C}^{2})$ consisting of
elements with positive topological entropy is an infinite cyclic group $%
\mathbb{Z}$.
\end{corollary}

\begin{proof}
Every polynomial automorphism of $\mathbb C^2$ is a proper homeomorphism and therefore extends to a homeomorphism of the one-point compactification \[ S^4=\mathbb C^2\cup\{\infty\}. \] Here $h_{\mathrm{top}}(f)$ denotes the topological entropy of this extension. For $f\in\operatorname{Aut}(\mathbb C^2)$, define \[ l_{\deg}(f) = \max\left\{ \lim_{n\to\infty}\frac{\log\deg(f^n)}{n}, \, \lim_{n\to\infty}\frac{\log\deg(f^{-n})}{n} \right\}. \] The limits exist because degree is submultiplicative. The function $l_{\deg}$ is a length function. Indeed, homogeneity follows directly from the definition. If $u\in\operatorname{Aut}(\mathbb C^2)$, then \[ \deg(uf^nu^{-1}) \leq \deg(u)\deg(f^n)\deg(u^{-1}), \] and the reverse inequality follows by applying the same estimate with $u^{-1}$. Hence $l_{\deg}$ is conjugacy invariant. If $f$ and $g$ commute, then \[ (fg)^n=f^ng^n, \] so \[ \deg((fg)^n) \leq \deg(f^n)\deg(g^n). \] The same inequality holds for the inverse iterates, and therefore \[ l_{\deg}(fg) \leq l_{\deg}(f)+l_{\deg}(g). \] It is well-known that an element conjugate into either $A$ or $E$ has zero topological entropy (cf. Friedland-Milnor \cite{fm}). If $f$ is  a cyclically reduced
element of amalgam length at least two in $A*_{A\cap E}E,$  then its degree is an integer at least two, its iterates have multiplicative degree growth, and Friedland--Milnor \cite{fm} together with Smillie's theorem \cite{sm} gives \[ l_{\deg}(f)=h_{\mathrm{top}}(f)=\log(\deg f)\geq\log2. \] It follows from the hypothesis on $S$ that the restriction of $l_{\deg}$ to $S$ is purely positive and discrete. Theorem~1.2 therefore implies that $S$ is virtually a finitely generated free abelian group. Consider the action of \[ \operatorname{Aut}(\mathbb C^2) = A*_{A\cap E}E \] on its Bass--Serre tree. A vertex stabilizer is conjugate to either $A$ or $E$. Every element of such a stabilizer has zero topological entropy. Since every non-identity element of $S$ has positive topological entropy, the induced action of $S$ has trivial vertex stabilizers. Hence $S$ acts freely on the Bass--Serre tree and is a free group. A nontrivial free group which is virtually abelian is infinite cyclic. Therefore $S\cong\mathbb Z.$
\end{proof}

\begin{remark}
We don't know whether there is an example $G$ consisting of elements with
positive topological entropy as in Theorem \ref{th-1} that is virtually
abelian-by-abelian, but not virtually abelian.
\end{remark}

\subsection{Group acting on CAT(0) cubical complexes, Helly graph and
injective metric space}

Let $X$ be a CAT(0) cubical complex and $\gamma :X\rightarrow X$ a cubical
isometry. An embedding $f:\mathbb{R}\rightarrow X$ is a combinatorial
geodesic if $f(n)$ is a $0$-cube and $f([n,n+1])$ is an $1$-cube for any
integer $n.$ The combinatorial translation length $l(\gamma )$ is defined as 
$\inf_{x\in X}d^{c}(x,\gamma x),$ where $d^{c}$ is the combinatorial length
defined by the minimal length combinatorial geodesic segment (see Haglund 
\cite{ha}). Haglund proves that any cubical isometry $\gamma $ either has a
fixed point in the $0$-skeleton or has a combinatorial geodesic $f(%
\mathbb{R})$ on which $\gamma $ acts as translation by distance  $%
l(\gamma )$ (on the cubical subdivision).

For any metric space $(Y,d),$ and isometry $\gamma :Y\rightarrow Y,$ there
is a standard way to define a length function, i.e. the stable translation
length, defined as the following. Fix $y\in Y,$ let $l_{0}(\gamma
)=\lim_{n\rightarrow \infty }d(y,\gamma ^{n}y)/n.$ It's easily seen that $%
l_{0}$ is independent of the choice of $y$ (see \cite{ye}, Section 2.5).
Note that%
\begin{equation*}
d(y,\gamma ^{n}y)\leq d(y,\gamma y)+d(\gamma y,\gamma ^{2}y)+\cdots
+d(\gamma ^{n-1}y,\gamma ^{n}y)=nd(y,\gamma y)
\end{equation*}%
implying that $l_{0}(\gamma )=\lim_{n\rightarrow \infty }d(y,\gamma
^{n}y)/n\leq d(y,\gamma y)$ and $l_{0}(\gamma )\leq \inf_{y\in Y}d(y,\gamma
y).$

\begin{definition}
An isometry $\gamma :Y\rightarrow Y$ is called semisimple, if either $\gamma 
$ has a bounded orbit, or there is a quasi-isometric orbit $n\rightarrow
\gamma ^{n}y$ for some $y\in Y.$
\end{definition}

For basic facts on Helly graphs and injective spaces, we refer the readers
to the lecture notes of Haettel \cite{he}.

\begin{lemma}
1) Let $G$ be a group acting properly on a CAT(0) cubical complex by cubical
automorphisms. The combinatorial translation length $l$ (on the cubical subdivision) is a discrete purely
positive length function on $G.$

2) Let $G$ be a group acting properly on a Helly graph $\Gamma $ of finite
combinatorial dimension. The translation length is a discrete and purely
positive.

3) Let $G$ be a group acting properly and cocompactly on an injective metric
space. The translation length is purely positive and discrete.
\end{lemma}

\begin{proof}
1) It is clear that 
\begin{equation*}
l(\gamma )=\inf_{x\in X}d^{c}(x,\gamma x)\geq 1
\end{equation*}%
when an infinite-order element $\gamma $ acts properly. Haglund proves that
when $\gamma $ doesn't fix any point, there is a combinatorial geodesic ray $%
f(\mathbb{R})$ on which $\gamma $ acts as translation (Corollary 5.2, \cite%
{ha}) of combinatorial distance $l(\gamma )$. Therefore, for any $y\in f(%
\mathbb{Z}),$ we have $l_{0}(\gamma )=\lim_{n\rightarrow \infty }d(y,\gamma
^{n}y)/n=l(\gamma ).$ This proves that $l$ is a length function, since $%
l_{0} $ is a length function (see Lemma 2.7, \cite{ye}).

2) By Haettel \cite{he} (Theorem 8.2), a non-elliptic $\gamma $ has an
integer $1\leq a\leq \dim (\Gamma ),L\geq 1$ and a vertex $v$ such that 
\begin{equation*}
d(\gamma ^{an}v,v)=|n|L,
\end{equation*}%
for any $n\in \mathbb{Z}.$ This implies that%
\begin{eqnarray*}
l(\gamma ^{a}) &=&\lim_{n\rightarrow \infty }d(v,\gamma ^{an}v)/n=L, \\
l(\gamma ) &=&L/a\geq 1/a.
\end{eqnarray*}%
Therefore, $\gamma $ is discrete and purely positive.

3) By Haettel \cite{he} (Theorem 8.1), a non-elliptic $\gamma $ has a
quasi-isometric embedding orbit: $n\in \mathbb{Z}\rightarrow \gamma ^{n}y$
for some $y\in Y.$ Therefore, there are numbers $K\geq 1$ and $C\geq 0$ such
that for any integer $n$ we have 
\begin{eqnarray*}
\frac{n}{K}-C &\leq &d(\gamma ^{n}y,y)\leq Kn+C, \\
l(\gamma ) &=&\lim_{n\rightarrow \infty }d(y,\gamma ^{n}y)/n\geq 1/K.
\end{eqnarray*}%
Therefore, $\gamma $ is purely positive. The discreteness is proved by
Abbott-Hagen-Petyt-Zalloum \cite{abt} (Theorem 1.2).
\end{proof}

Haglund \cite{ha} (Corollary 1.7) proves that a Heisenberg group has no
proper action on a discrete space with walls (or a CAT0) cubical complex).
Genevois \cite{gev} (Theorem 1.5 and Corollary 1.6) proves that a polycyclic
group acting properly on a CAT(0) cube complex is virtually abelian. These
results are special cases of the following 1).

\begin{corollary}
Let a finitely generated solvable group $G$ act properly on a CAT(0) cubical
complex.

1) If the virtual cohomological dimension of $G$ is finite (eg. polycyclic
group), then $G$ is virtually $\mathbb{Z}^{n}$ for some integer $n.$

2) If $G$ is virtually torsion-free, and any abelian subgroup $A$ has finite 
$\dim _{\mathbb{Q}}(A\bigotimes_{\mathbb{Z}}\mathbb{Q})<+\infty $, then $G$
is virtually either abelian or a non-nilpotent $\mathbb{C}$-linear subgroup
of $H\rtimes \mathbb{Z}^{n}$ for some integer $n$ and a finitely generated $%
\mathbb{Z}[\mathbb{Z}^{n}]$-module $H.$
\end{corollary}

\begin{proof}
Since $l$ is a discrete length function, 1) follows the solvable subgroup
theory Theorem \ref{th3}. The second claim 2) follows Theorem \ref{th1}.
\end{proof}

\subsection{Virtual cyclic retract groups and RFRS groups}

Recall that an RFRS group $G$ defined by Agol \cite{ag} is a finitely generated
group such that there is a sequence of finite-index normal subgroups $%
G=G_{0}>G_{1}>...$ such that 1) $\cap G_{i}=1$ 2) $G_{i+1}>\ker
(G_{i}\rightarrow H_{1}(G_{i})\bigotimes \mathbb{Q})$ for each $i.$ T.
Koberda \cite{kob} (Theorem 6) shows that every virtually RFRS nilpotent
group is virtually abelian. Fisher \cite{fish} (Theorem D) furthermore
proves that amenable virtually RFRS groups of type $\mathrm{FP}(\mathbb{Q})$
are virtually Abelian. A group $G$ is said to be VRC (virtual cyclic
retract) if for every cyclic subgroup $\mathbb{Z}$, there is a finite-index
subgroup $H<G$ such that $H>\mathbb{Z}$ and $H$ has an epimorphism $%
H\rightarrow \mathbb{Z}$ satisfying $\mathbb{Z}\rightarrow H\rightarrow 
\mathbb{Z}$ is identity. Minasyan \cite{Min} (Theorem 9.1) proves that a
virtually polycyclic VRC group is virtually abelian and remarks that this
holds as well for solvable groups which are virtually torsion-free and all
of whose abelian subgroups have finite $\mathbb{Q}$-ranks (see Remark 9.3).

\begin{lemma}
\label{vrc}
The stable word length of a finitely generated (VRC) group (eg. virtually
RFRS group) $G$ is purely positive. Actually, any finitely generated residually $\mathbb{Z}$ group has a purely positive stable word length.
\end{lemma}

\begin{proof}
Let $G$ be finitely generated, with finite generating set $S$, and suppose that $G$ is virtually residually $\mathbb{Z}$. We claim that every infinite-order element $ g\in G$ has a positive stable word length. By hypothesis, there exist a finite-index subgroup $H<G$, an integer $n\geq 1$, and an epimorphism $f\to\mathbb Z$ such that $g^n\in H$ and $f(g^n)\neq 0$. Choose a finite generating set $T$ for $H$. Since $H$ has finite index in $G$, the inclusion $H\hookrightarrow G$ is a quasi-isometric embedding, implying that hence there is a constant $C>0$ such that $|h|_T\leq C|h|_S$ for all $h\in H$. Since $f$ is a homomorphism, there is $L>0$ such that $|f(h)|\leq L|h|_T$ for all $h\in H$. Thus, for every $k\geq1$,
\[
k|f(g^n)|
=|f(g^{nk})|
\leq LC|g^{nk}|_S.
\]
Dividing by $nk$ and letting $k\to\infty$ gives
\[
\lim_{k\to\infty}\frac{|g^{nk}|_S}{nk}
\geq \frac{|f(g^n)|}{nLC}>0.
\]
Hence every nontrivial element of $G$ has positive stable word length.

Note that in a VRC group every infinite-order element virtually survives in a $\mathbb{Z}$-quotient. It's already proved by
Fisher \cite{fish} (Prop. 2.7) that any infinite cyclic subgroup of a
virtually RFRS group $G$ is a virtual retractor. This means a virtually RFRS
group has property (VRC). Therefore, the proof is complete.
\end{proof}

The solvable subgroup theorem obtained from Theorem \ref{th1} and Lemma \ref{vrc} applied to VRC
or RFRS is not stronger than the above-mentioned results of Minasyan's and
Fisher's, as we are not sure whether the stable word length in Lemma \ref{vrc} is always discrete (to apply Theorem \ref{th3}). However, we believe that the Lemma \ref{vrc} is of independent interest.  

\subsection{Interval Exchange Transformations}

Let $D=[0,1)$ or $S^{1}.$ A bijective map $f:D\rightarrow D$ is called an
Interval Exchange Transformation if there is a partition%
\begin{equation*}
0=x_{1}<x_{2}<...<x_{n}=1
\end{equation*}%
such that for each $i$ the restriction $f|_{[x_{i},x_{i+1})}$ is an
orientation-preserving isometry onto $[x_{s},x_{s+1})$ for some $s.$ Let $%
\mathrm{IET}(D)$ be the group of interval exchange transformations. Let $%
d(f) $ be the number of discontinuous points and 
\begin{equation*}
|f|=\lim \frac{d(f^{n})}{n}.
\end{equation*}%
Dahmani, Fujiwara and Guirardel \cite{dfg} (Lemma 2.2) prove that $d(fg)\leq
d(f)+d(g)$. For the concept of virtual multi-rotation, see \cite{dfg}
(Section 1.2). Dahmani, Fujiwara and Guirardel \cite{dfg} and Novak \cite%
{nov} have proved the following results.

\begin{lemma}
\label{iet}1) The $|-|:\mathrm{IET}(D)\rightarrow \mathbb{N}_{\geq 0}$ is a
length function.

2) If $S$ is a finite set of elements in $\mathrm{IET}(D)$, then the group
generated by $S$ contains no distorted element. In other words, the stable
word length on the subgroup $\langle S\rangle $ is purely positive.

3) $|h|=0$ if and only if $h$ is conjugate to a virtual multi-rotation.
\end{lemma}

\begin{proof}
1) By Lemma 2.2 of Dahmani, Fujiwara and Guirardel \cite{dfg}, $d$ is
subadditive. Therefore, the function $|-|$ is a length function by \cite{ye}
(Lemma 2.1). Corollary 2.5. i) of \cite{dfg} shows that $|f|$ is an integer.

2) is proved by Novak \cite{nov} (Theorem 1.3), while 3) is Corollary 2.5.
4) of \cite{dfg}.
\end{proof}

Dahmani, Fujiwara and Guirardel \cite{dfg2} (Theorem 3.1) prove that every
finitely generated torsion-free solvable subgroup of $\mathrm{IET}(D)$ is
virtually abelian. This also follows Theorem \ref{th3} and Theorem \ref%
{th1}, since  the length function in Lemma \ref{iet} 2)
is an upper bound of the stable number $|f|$ of discontinuous points, which is proved to be discrete by Cornulier \cite{cor} (Corollary 6.4). On the other hand, Theorem \ref{th3} implies the following.

\begin{corollary}
Any subgroup $Q<\mathrm{IET}(D)$, which is isomorphic to the additive group $(\mathbb Q,+),$ contains a nontrivial
virtual multi-rotation (up to conjugation).
\end{corollary}

\begin{proof}
Suppose that $Q$ contains no nontrivial element conjugate to a virtual multi-rotation. By Lemma~5.8, the discontinuity-growth function \[ |\cdot|\colon\operatorname{IET}(D)\longrightarrow\mathbb N_{\geq0} \] is then positive on every non-identity element of $Q$. Since it is integer-valued, its restriction to $Q$ is a discrete purely positive length function. The additive group $\mathbb Q$ has \[ \operatorname{cd}_{\mathbb Z}(\mathbb Q)=2. \] Therefore Theorem~1.2 would imply that $Q$ is virtually a finitely generated free abelian group. This is impossible. Indeed, the divisible group $\mathbb Q$ has no proper finite-index subgroup: if $H<\mathbb Q$ has finite index $m$, then \[ m\mathbb Q\leq H, \] but \[ m\mathbb Q=\mathbb Q. \] Thus $H=\mathbb Q$, whereas $\mathbb Q$ is not finitely generated. This contradiction proves the corollary.
[If a subgroup $H<\mathrm{IET}(D)$ has no virtual multi-rotations, Lemma \ref%
{iet} implies that $|-|$ is discrete and purely positive on $H$. Theorem \ref%
{th3} implies that such $H=\mathbb{Q}\ $would be virtually $\mathbb{Z}^{n},$
which is absurd.]
\end{proof}

More stronger results are proved by Cornulier \cite{cor}.

\noindent \textbf{Acknowledgments}

The author expresses gratitude to the referees for providing valuable
comments on a prior version, pointing out the length function on IET is discrete, with special acknowledgment for the enhancement
made to Corollary \ref{last}. We also thank Prof. Greg Conner for helpful
communications. This work is supported by NSFC (No. 11971389).

\bigskip

NYU Shanghai, No.567 Yangsi West Rd, Pudong New Area, Shanghai, 200124,
China.

NYU-ECNU Institute of Mathematical Sciences at NYU Shanghai, 3663 Zhongshan
Road North, Shanghai, 200062, China.

E-mail: sy55@nyu.edu

\end{document}